\documentclass[12pt]{amsart}
\usepackage{amssymb}
\usepackage[all]{xy}

\swapnumbers 
\newtheorem{thm}[subsection]{Theorem}
\newtheorem{lem}[subsection]{Lemma}
\newtheorem{prop}[subsection]{Proposition}

\theoremstyle{definition}

\newtheorem*{example*}{Example}

\theoremstyle{remark}
\newtheorem{rmk}[subsection]{Remark}
\numberwithin{equation}{subsection}

\newcommand{\N}{{\mathbb N}}
\newcommand{\Z}{{\mathbb Z}}
\newcommand{\Q}{{\mathbb Q}}

\newcommand{\Hom}{\operatorname{Hom}}

\newcommand{\divided}[2]{#1^{(#2)}}
\newcommand{\sqbinom}[2]{\genfrac{[}{]}{0pt}{}{#1}{#2}}
\newcommand{\UU}{\mathbf{U}}
\newcommand{\Sch}{\mathbf{S}}
\renewcommand{\P}{\mathbf{P}}
\newcommand{\B}{\mathbf{B}}
\newcommand{\A}{\mathcal{A}}
\newcommand{\gl}{\mathfrak{gl}}
\newcommand{\catC}{\mathcal{C}}
\newcommand{\bil}[2]{\langle #1, #2 \rangle}

\renewcommand{\ge}{\geqslant}
\renewcommand{\le}{\leqslant}
\newcommand{\phat}{\widehat{p}}
\newcommand{\pdot}{\dot{p}}

\begin{document}
\title[quantized enveloping algebras via inverse limits]
{Constructing quantized enveloping algebras via inverse limits of
  finite dimensional algebras}

\author[S.~Doty]{Stephen Doty}
\address{Department of Mathematics and Statistics\\ 
Loyola University Chi\-cago\\ 
Chicago, Illinois 60626 USA}
\thanks{Partly supported by the 2007 Yip Fellowship at Magdalene College,
  Cambridge.}
  \email{doty@math.luc.edu}

\date{12 July 2008}

\begin{abstract} It is known that a generalized $q$-Schur algebra
  may be constructed as a quotient of a quantized enveloping algebra
  $\UU$ or its modified form $\dot{\UU}$.  On the other hand, we show
  here that both $\UU$ and $\dot{\UU}$ may be constructed within an
  inverse limit of a certain inverse system of generalized $q$-Schur
  algebras. Working within the inverse limit $\widehat{\UU}$ clarifies
  the relation between $\dot{\UU}$ and $\UU$. This inverse limit is a
  $q$-analogue of the linear dual $R[G]^*$ of the coordinate algebra
  of a corresponding linear algebraic group $G$.
\end{abstract}
\maketitle

\parskip=4pt 
\allowdisplaybreaks 

\section*{Introduction}\noindent 
Beilinson, Lusztig, and MacPherson \cite{BLM} constructed a quantized
enveloping algebra $\UU$ corresponding to the general linear Lie
algebra $\gl_n$ within the inverse limit of an inverse system
constructed from $q$-Schur algebras. The modified form $\dot{\UU}$ of
$\UU$ was also obtained within the inverse limit.  Using a slightly
different inverse system, consisting of all the generalized $q$-Schur
algebras connected to a given root datum, we construct both $\UU$ and
$\dot{\UU}$ as subalgebras of the resulting inverse limit.  This
approach, which is analogous to the inverse limit construction of
profinite groups, works uniformly for any root datum of finite type,
not just for type $A$. In particular, this clarifies the relation
between $\UU$ and $\dot{\UU}$. All of this is for the generic case,
i.e., working over the field $\Q(v)$ of rational functions in $v$, $v$
an indeterminate.  However, the construction is compatible with the
so-called ``restricted'' integral form of Lusztig, and (in a certain
sense made precise in \S5) is also compatible with specializations
defined in terms of the restricted integral form.

Generalized Schur algebras were introduced by Donkin \cite{Donkin:SA},
motivated by \cite{Green:book}. In \cite{PGSA} a uniform system of
generators and relations was found for them and their $q$-analogues
(this was known earlier \cite{DG:PSA} in type $A$) and it was proved
that the generalized $q$-Schur algebras are quasihereditary in any
specialization to a field. The generators and relations of \cite{PGSA}
allow a definition of the generalized $q$-Schur algebras independent
of the theory of quantized enveloping algebras; they also lead
directly to the inverse system considered here.

Similar inverse systems appeared in \cite{RM-Green} (in a more general
context) and in \cite{Liu} (in the classical case for types A--D). It
turns out that the inverse limit we construct is a ``procellular''
completion of $\dot{\UU}$, in the sense of \cite{RM-Green}. In
particular, its elements may be described as formal, possibly
infinite, linear combinations of the canonical basis of $\dot{\UU}$.

\section{Notation}\noindent
We fix our notational conventions, which are similar to those of
\cite{Lusztig}.

\subsection{Cartan datum} 
Let a Cartan datum be given. By definition, a Cartan datum consists of
a finite set $I$ and a symmetric bilinear form $(\ ,\ )$ on the free
abelian group $\Z[I]$ taking values in $\Z$, such that:

(a) $(i,i) \in \{2,4,6,\dots \}$ for any $i$ in $I$. 

(b) $2(i,j)/(i,i) \in \{0, -1, -2, \dots \}$ for any $i \ne j$ in $I$.

\subsection{Root datum}\label{root}
A root datum associated to the given Cartan datum consists of two
finitely generated free abelian groups $X$, $Y$ and a
perfect\footnote{A `perfect' pairing is one for which the natural maps
  $X \to \Hom_\Z(Y,\Z)$ (given by $x \to \bil{-}{x}$) and $Y
  \to \Hom_\Z(X,\Z)$ (given by $y \to \bil{y}{-}$) are
  isomorphisms.}  bilinear pairing $\bil{\ }{\ }: Y \times X \to \Z$
along with embeddings $I \to Y$ ($i \mapsto h_i$) and $I \to X$ ($i
\mapsto \alpha_i$) such that
\[
   \bil{h_i}{\alpha_j} = 2\frac{(i,j)}{(i,i)}
\]
for all $i,j$ in $I$.  The image of the embedding $I \to Y$ is the set
$\{h_i\}$ of simple coroots and the image of the embedding $I \to X$
is the set $\{\alpha_i\}$ of simple roots.

\subsection{}
The assumptions on the root datum imply that

(a) $\bil{h_i}{\alpha_i} = 2$ for all $i\in I$;

(b) $\bil{h_i}{\alpha_j} \in \{0, -1, -2, \dots \}$ for all $i \ne j
\in I$.

\noindent
In other words, the matrix $(\, \bil{h_i}{\alpha_j} \,)$ indexed by $I
\times I$ is a symmetrizable generalized Cartan matrix.

For each $i \in I$ we set $d_i = (i,i)/2$. Then the matrix $(\,d_i
\bil{h_i}{\alpha_j} \,)$ indexed by $I \times I$ is symmetric.

\subsection{}
Let $v$ be an indeterminate. Set $v_i = v^{d_i}$ for each $i \in
I$. More generally, given any rational function $P \in \Q(v)$ we let
$P_i$ denote the rational function obtained from $P$ by replacing $v$
by $v_i$.

Set $\A = \Z[v,v^{-1}]$. For $a\in \Z$, $t\in \N$ we set 
\[
   \sqbinom{a}{t} = \prod_{s=1}^t \frac{v^{a-s+1} - v^{-a+s-1}} {v^s -
   v^{-s}}.
\]
A priori this is an element of $\Q(v)$, but actually it lies in $\A$
(see \cite[\S1.3.1(d)]{Lusztig}). We set
\[
  [n] = \sqbinom{n}{1} = \frac{v^n - v^{-n}}{v - v^{-1}}
\qquad (n \in \Z)
\]
and 
\[
   [n]^! = [1]\cdots [n-1]\,[n] \qquad (n\in \N).
\]
Then it follows that
\[
\sqbinom{a}{t} = \frac{[a]^!}{[t]^!\,[a-t]^!}
\qquad \text{for all $0 \le t \le a$.}
\]

\subsection{} \label{terminology}
The Cartan datum is of {\em finite type} if the symmetric matrix $(\,
(i,j)\, )$ indexed by $I\times I$ is positive definite. This is
equivalent to the requirement that the Weyl group associated to the
Cartan datum is a finite group.

A root datum is $X$-{\em regular} (resp., $Y$-{\em regular}) if
$\{\alpha_i\}$ (resp., $\{h_i\}$) is linearly independent in $X$
(resp., $Y$). If the underlying Cartan datum is of finite type then
the root datum is automatically both $X$-regular and $Y$-regular.

In case a root datum is $X$-regular, there is a partial order on $X$
given by: $\lambda \le \lambda'$ if and only if $\lambda' - \lambda
\in \sum_i \N \alpha_i$.  In case a root datum is $Y$-regular, we
define
\[
X^+ = \{ \lambda \in X \mid \bil{h_i}{\lambda} \in \N, 
\text{ for all } i \in I \},
\]
the set of dominant weights.

\subsection{} \label{catC} 
Corresponding to a given root datum is a quantized enveloping algebra
$\UU$ over $\Q(v)$.  According to \cite[Corollary 33.1.5]{Lusztig},
the algebra $\UU$ is the associative algebra with 1 over $\Q(v)$ given
by the generators $E_i$, $E_{-i}$ ($i\in I$), $K_h$ ($h \in Y$)
subject to the defining relations

(a) $K_h K_{h'} = K_{h+h'}$; \quad $K_0 = 1$; 

(b) $K_h E_{\pm i} = v^{\pm\bil{h}{\alpha_i}} E_{\pm i} K_h$; 

(c) $E_i E_{-j} - E_{-j} E_i = \delta_{ij} \dfrac{\widetilde{K}_{i} -
  \widetilde{K}_{-i}}{v_i - v_i^{-1}}$ \quad where $\widetilde{K}_{\pm i}:=
K_{\pm d_ih_i}$;

(d) $\displaystyle \sum_{s+s'=1-\bil{h_i}{\alpha_j}} (-1)^{s'}
\sqbinom{1-\bil{h_i}{\alpha_j}}{s}_i E_{\pm i}^s E_{\pm j} E_{\pm
  i}^{s'} = 0$\quad ($i \ne j$)

\noindent
holding for any $i,j \in I$, and any $h, h' \in Y$. 

We define, for an element $E$, the quantized divided power
$\divided{E}{m}$ of $E$ by
\[
\divided{E}{m}:= \frac{E^m}{[m]^!_i}
\]
for any $m \in \N$. With this convention, one may rewrite relation (d)
in the equivalent form

($\mathrm{d}'$) $\displaystyle \sum_{s+s'=1-\bil{h_i}{\alpha_j}}
  (-1)^{s'} \divided{E_{\pm i}}{s} E_{\pm j} \divided{E_{\pm i}}{s'} =
  0$\quad ($i \ne j$).

As in \cite[\S3.4]{Lusztig}, let $\catC$ be the category whose
objects are $\UU$-modules $M$ admitting a weight space decomposition
$M = \oplus_{\lambda \in X} M_\lambda$ (as $\Q(v)$-vector spaces)
where the weight space $M_\lambda$ is given by
\[
  M_\lambda = \{ m\in M \mid K_h m = v^{\bil{h}{\lambda}} m, \text{
  all } h \in Y \}.
\]
The morphisms in $\catC$ are $\UU$-module homomorphisms. 

From now on we assume the root datum is of finite type. Thus it is
both $X$- and $Y$-regular.  We denote by $\Delta(\lambda)$ the simple
object (see \cite[Cor.~6.2.3, Prop.~3.5.6]{Lusztig}) of $\mathcal{C}$
of highest weight $\lambda \in X^+$, for any $\lambda \in X^+$.

\section{The algebra $\widehat{\UU}$}\noindent 
Fix a root datum $(X,\{\alpha_i\}, Y,\{h_i\})$ of {\em finite
  type}. We define the finite dimensional algebras $\Sch(\pi)$ (the
generalized $q$-Schur algebras) and construct the algebra
$\widehat{\UU}$ as an inverse limit.

\subsection{}
A nonempty subset $\pi$ of $X^+$ is {\em saturated} if $\lambda \le
\mu$ for $\lambda \in X^+$, $\mu \in \pi$ implies $\lambda \in \pi$.

Saturated subsets of $X^+$ exist in abundance. For instance, given any
$\mu \in X^+$, the set $X^+[\le \mu] = \{ \lambda \in X^+ \mid \lambda
\le \mu \}$ is saturated. In general, a saturated subset of $X^+$ is a
union of such subsets.

\subsection{The algebra $\Sch(\pi)$}\label{def:Spi}
Given a finite saturated set $\pi \subset X^+$ we define an algebra
$\Sch(\pi)$ to be the associative $\Q(v)$-algebra with 1 given by
the generators
\[
 E_i, E_{-i} \quad(i \in I),\qquad 1_\lambda \quad(\lambda \in W\pi)
\]
and the relations

(a) $1_\lambda 1_{\lambda'} = \delta_{\lambda,\lambda'} 1_\lambda$, \qquad
$\sum_{\lambda \in W\pi} 1_\lambda = 1$;

(b) $E_{\pm i} 1_\lambda = 
\begin{cases}
1_{\lambda \pm \alpha_i} E_{\pm i} & \text{if } \lambda \pm \alpha_i
\in W\pi \\ 0 & \text{otherwise};
\end{cases}$

($\mathrm{b}'$) $1_\lambda E_{\pm i} = 
\begin{cases}
E_{\pm i} 1_{\lambda \mp \alpha_i} & \text{if } \lambda \mp \alpha_i
\in W\pi \\ 0 & \text{otherwise}; 
\end{cases}$

(c) $E_iE_{-j} - E_{-j}E_i = \delta_{ij} \sum_{\lambda \in W\pi}
\;[\bil{h_i}{\lambda}]_i \, 1_\lambda$;

(d) $\displaystyle \sum_{s+s'=1-\bil{h_i}{\alpha_j}} (-1)^{s'}
\divided{E_{\pm i}}{s} E_{\pm j} \divided{E_{\pm i}}{s'} =
  0$\quad  ($i \ne j$)

\noindent
for all $i,j \in I$ and all $\lambda, \lambda' \in W\pi$. In relation
(d), $\divided{E_{\pm i}}{s}$ is the quantized divided power
$\divided{E_{\pm i}}{s} = E_{\pm i}^s/([s]_i^!)$.

The algebra $\Sch(\pi)$ is known as a generalized $q$-Schur algebra
(see \cite{PGSA}). It is a consequence of the defining relations that
the generators $E_{\pm i}$ are nilpotent elements of $\Sch(\pi)$; it
follows that $\Sch(\pi)$ is finite dimensional over $\Q(v)$.

For any $\pi$ we define elements $K_h \in \Sch(\pi) $ for each $h \in
Y$ by the formula
\[
K_h = \textstyle \sum_{\lambda \in W\pi} v^{\bil{h}{\lambda}}
1_\lambda .
\]
This depends on $\pi$ as well as $h$; we rely on the context to make
clear in which $\Sch(\pi)$ a given $K_h$ is to be interpreted.  We
note that the identities $K_h K_{h'} = K_{h+h'}$, $K_0 = 1$ and
$K_{-h} = K_h^{-1}$ hold in $\Sch(\pi)$ for all $h, h' \in Y$.

\subsection{} \label{Spi:reformulation}
It will be convenient for ease of notation to extend the meaning of
the symbols $1_\lambda$ to all $\lambda \in X$ by making the
convention $1_\lambda = 0$ in $\Sch(\pi)$ for any $\lambda \notin
W\pi$. With this convention $\Sch(\pi)$ becomes the associative
$\Q(v)$-algebra given by generators
\[
 E_{i}, E_{-i} \quad(i \in I),\qquad 1_\lambda \quad(\lambda \in W\pi)
\]
with the relations

(a) $1_\lambda 1_{\lambda'} = \delta_{\lambda,\lambda'} 1_\lambda$,
\qquad $\sum_{\lambda \in X} 1_\lambda = 1$;

(b) $E_{\pm i} 1_\lambda = 1_{\lambda \pm \alpha_i} E_{\pm i}$;

(c) $E_{i}E_{-j} - E_{-j} E_{i} = \delta_{ij} \sum_{\lambda \in X}
\;[\bil{h_i}{\lambda}]_i \, 1_\lambda$;

(d) $\displaystyle \sum_{s+s'=1-\bil{h_i}{\alpha_j}} (-1)^{s'}
\divided{E_{\pm i}}{s} E_{\pm j} \divided{E_{\pm i}}{s'} = 0$\quad 
($i \ne j$)

\noindent
for all $i,j \in I$ and all $\lambda, \lambda' \in X$. Note that the
sums in (a), (c) are finite since by definition all but finitely many
$1_\lambda$ are zero in $\Sch(\pi)$.

\subsection{}
The form of the presentation of $\Sch(\pi)$ given in
\ref{Spi:reformulation} makes it clear that for any finite saturated
subsets $\pi, \pi'$ of $X^+$ with $\pi \subset \pi'$ we have a
surjective algebra map
\[
f_{\pi,\pi'}: \Sch(\pi') \to \Sch(\pi)
\] 
sending $E_{\pm i} \to E_{\pm i}$, $1_\lambda \to 1_\lambda$ (any $i
\in I$, $\lambda \in W\pi'$). Since $f_{\pi,\pi} = 1$ and for any
finite saturated subsets $\pi, \pi', \pi''$ of $X^+$ with $\pi \subset
\pi' \subset \pi''$ we have $f_{\pi,\pi'} f_{\pi',\pi''} =
f_{\pi,\pi''}$, the collection
\[
  \{ \Sch(\pi) ; f_{\pi,\pi'} \}
\]
forms an inverse system of algebras. We denote by $\widehat{\UU} =
\varprojlim \Sch(\pi)$ the inverse limit of this inverse system,
taken over the collection of all finite saturated subsets of $X^+$.
This is isomorphic with 
\[
\{\textstyle (a_\pi)_\pi \in \prod_\pi \Sch(\pi) \mid a_\pi =
f_{\pi,\pi'}(a_{\pi'}), \text{ for any $\pi \subset \pi'$} \}
\]
with addition and multiplication of such sequences defined
componentwise.  We set 
\[
\widehat{1}_\lambda := (1_\lambda)_\pi  \quad \in \widehat{\UU}
\]
and note that because of the convention introduced in
\ref{Spi:reformulation} a number of the components of this sequence
may be zero. However, only finitely many components are zero, so the
sequence is eventually constant. We similarly set
\[
\widehat{E}_{\pm i} := (E_{\pm i})_\pi  \quad \in \widehat{\UU}
\]
for any $i \in I$. Finally, for any $h \in Y$ we set 
\[
\widehat{K}_h := (K_h)_\pi \quad \in \widehat{\UU}.
\]

\subsection{}\label{comm-diag}
Let $\phat_\pi: \widehat{\UU} \to \Sch(\pi)$ be projection onto the
$\pi$th component.  Let $\UU$ be the quantized enveloping algebra
determined by the given root datum (see \ref{catC}) and for each
$\lambda \in X^+$ let $\Delta(\lambda)$ be the simple $\UU$-module of
highest weight $\lambda$.  According to \cite[Corollary~3.13]{PGSA},
$\Sch(\pi)$ is the quotient of $\UU$ by the ideal consisting of all $u
\in \UU$ annihilating every simple module $\Delta(\lambda)$ such that
$\lambda \in \pi$. Let $p_\pi: \UU \to \Sch(\pi)$ be the corresponding
quotient map, which sends $E_{\pm i}\in \UU$ to $E_{\pm i} \in
\Sch(\pi)$, $K_h \in \UU$ to $K_h \in \Sch(\pi)$ for all $i \in I$, $h
\in Y$. These maps fit into a commutative diagram
\[
\xymatrix{
\widehat{\UU} \ar[dr]^{\phat_{\pi'}} \ar@/_/[dddr]_{\phat_\pi} & & 
\UU \ar[dl]_{p_{\pi'}} \ar@/^/[dddl]^{p_\pi}
\ar@{.>}[ll]_\theta \\ 
& \Sch(\pi') \ar[dd]|{f_{\pi,\pi'}} \\ \\
& \Sch(\pi) 
}
\]
for any finite saturated subsets $\pi, \pi'$ of $X^+$ with $\pi
\subset \pi'$. The universal property of inverse limits guarantees the
existence of a unique algebra map $\theta: \UU \to \widehat{\UU}$
making the diagram commute. 

\begin{thm} \label{thm:Uembeds}
The map $\theta$ gives an algebra embedding of $\UU$ into
$\widehat{\UU}$ sending $E_{\pm i}$ to $\widehat{E}_{\pm i}$ and $K_h$
to $\widehat{K}_h$ for all $i \in I$, $h \in Y$. Hence, the subalgebra
of $\widehat{\UU}$ generated by the $\widehat{E}_{\pm i}$ ($i \in I$),
$\widehat{K}_h$ ($h \in Y$) is isomorphic with $\UU$.
\end{thm} 

\begin{proof}
Suppose $u \in \UU$ maps to zero under $\theta$. Then $p_\pi(u) = 0$
for every finite saturated subset $\pi$, which means that $u$
annihilates every simple $\UU$-module in the category
$\mathcal{C}$. By \cite[Prop.~3.5.4]{Lusztig} it follows that $u = 0$,
so the kernel of $\theta$ is trivial.  The rest of the assertions of
the theorem are clear.
\end{proof}

\begin{prop}\label{prop:K_h}
In $\widehat{\UU}$ we have the identity $\widehat{K}_h = \sum_{\lambda
  \in X} v^{\bil{h}{\lambda}} \,\widehat{1}_\lambda$ for any $h \in Y$.
\end{prop}

\begin{proof}
We have only to check that this holds when the projection $\phat_\pi$ is
applied to both sides. This is valid by the definition of $K_h \in
\Sch(\pi)$ given in \ref{def:Spi}.
\end{proof}

\subsection{}
From the preceding result it follows by easy calculations that in
$\widehat{\UU}$ we have the identities

(a) $\widehat{K}_h \widehat{K}_{h'} = \widehat{K}_{h+h'}$, 

(b) $\widehat{K}_0 = 1$, 

(c) $\widehat{K}_{-h} = \widehat{K}_h^{-1}$ 

\noindent
for any $h,h' \in Y$.

\begin{prop}\label{prop:Uhat-rels}
The elements $\widehat{E}_{\pm i}\ (i \in I)$;
$\widehat{1}_\lambda\ (\lambda \in X)$ of $\widehat{\UU}$ satisfy the
relations

(a) $\widehat{1}_\lambda \widehat{1}_{\lambda'} =
\delta_{\lambda,\lambda'} \widehat{1}_\lambda$, \qquad $\sum_{\lambda
  \in X} \widehat{1}_\lambda = 1$;

(b) $\widehat{E}_{\pm i} \widehat{1}_\lambda = \widehat{1}_{\lambda
  \pm \alpha_i} \widehat{E}_{\pm i}$;

(c) $\widehat{E}_{i}\widehat{E}_{-j} - \widehat{E}_{-j}
\widehat{E}_{i} = \delta_{ij} \sum_{\lambda \in X}
\;[\bil{h_i}{\lambda}]_i \, \widehat{1}_\lambda$;

(d) $\displaystyle \sum_{s+s'=1-\bil{h_i}{\alpha_j}} (-1)^{s'}
\divided{\widehat{E}_{\pm i}}{s} \widehat{E}_{\pm j}
\divided{\widehat{E}_{\pm i}}{s'} = 0$\quad  ($i \ne j$).
\end{prop}

\begin{proof}
The argument is similar to the proof of Proposition \ref{prop:K_h}.
\end{proof}

Note that the relations in the preceding result are the same relations
as in \ref{Spi:reformulation} but in this case the sums in (a), (c)
are infinite, since $1_\lambda \in \widehat{\UU}$ is nonzero for any
$\lambda \in X$.

\begin{rmk} \label{rmk:Urels}
It is clear from Theorem \ref{thm:Uembeds} that the elements
$\widehat{E}_{\pm i}$, $\widehat{K}_h$ of $\widehat{\UU}$ satisfy the
usual defining relations (see \ref{catC}(a)--(d)) for the quantized
enveloping algebra $\UU$.

On the other hand, if one simply starts with $\Sch(\pi)$ defined by
the presentation given in \ref{def:Spi} and forms the inverse limit
$\widehat{\UU}$ without knowledge of $\UU$, defining elements
$\widehat{E}_{\pm i}$, $\widehat{K}_h$ in $\widehat{\UU}$ as we have
done above, then the defining relations \ref{catC}(a)--(d) (with
$\widehat{E}_{\pm i}$, $\widehat{K}_h$ in place of $E_{\pm i}$, $K_h$
respectively) may easily be derived from the defining relations for
$\Sch(\pi)$.  Then $\UU$ could be defined as the subalgebra of
$\widehat{\UU}$ generated by the $\widehat{E}_{\pm i}$ ($i \in I$),
$\widehat{K}_h$ ($h \in Y$).  In fact, this is clear from Proposition
\ref{prop:Uhat-rels} in light of Proposition \ref{prop:K_h}.  For
instance, one has 
\begin{align*} 
\widehat{K}_h \widehat{E}_{\pm i} &= \textstyle\sum_\lambda
v^{\bil{h}{\lambda}} 1_\lambda \widehat{E}_{\pm i} = \sum_\lambda
v^{\bil{h}{\lambda}} \widehat{E}_{\pm i} 1_{\lambda \mp \alpha_i}\\ &=
\widehat{E}_{\pm i} \textstyle\sum_\lambda v^{\bil{h}{\lambda\pm
    \alpha_i}} 1_{\lambda} = v^{\pm\bil{h}{\alpha_i}} \widehat{E}_{\pm
  i} \widehat{K}_h
\end{align*}
 where in the sums $\lambda$ runs over $X$. This proves the analogue
 of relation \ref{catC}(b).  The analogue of relation \ref{catC}(c) is
 proved by a similar calculation, which we leave to the reader.  In
 other words, the defining structure of the quantized enveloping
 algebra $\UU$ is an easy consequence of the defining structure for
 the $\Sch(\pi)$.
\end{rmk}

\begin{rmk}
  The inverse system used here is indexed by the family consisting of
  all finite saturated subsets of $X^+$. One could just as well have
  used the family consisting of all subsets of the form $X^+[\le
  \lambda] = \{\mu \in X^+: \mu \le \lambda\}$, for various $\lambda
  \in X^+$, or even the family of complements of all the $X^+[\ge
  \lambda] = \{\mu \in X^+: \mu \ge \lambda\}$. All these families of
  finite saturated subsets of $X^+$ lead to the same inverse limit
  $\widehat{\UU}$.
\end{rmk}

\section{Relation with the modified form $\dot{\UU}$}\noindent
In this section we explore the relation between the algebra
$\widehat{\UU}$ and Lusztig's modified form $\dot{\UU}$ of $\UU$.  We
show that $\dot{\UU}$ may be identified with a subalgebra of
$\widehat{\UU}$.

\subsection{}\label{defUdot}
The modified form $\dot{\UU}$ is defined (see
\cite[Chapter~23]{Lusztig}) as follows. For $\lambda, \lambda' \in X$
set
\[
  {}_\lambda \UU_{\lambda'} = \UU/ \Big( \sum_{h \in Y} (K_h-
  v^{\bil{h}{\lambda}})\UU + \sum_{h \in Y} \UU(K_h-
  v^{\bil{h}{\lambda'}}) \Big)
\]
regarded as a quotient of vector spaces over $\Q(v)$. Then define
\[
\dot{\UU} := \textstyle \bigoplus_{\lambda, \lambda' \in X}
({}_\lambda \UU_{\lambda'}).
\]
Let $\pi_{\lambda, \lambda'}: \UU \to {}_\lambda \UU_{\lambda'}$ be
the canonical projection. One has a direct sum decomposition $\UU =
\bigoplus_\nu \UU(\nu)$ where $\nu$ runs over the root lattice $\sum
\Z\alpha_i$, and where $\UU(\nu)$ is defined by the requirements
$\UU(\nu)\UU(\nu') \subseteq \UU(\nu+\nu')$, $K_h \in \UU(0)$, $E_{\pm
  i} \in \UU(\pm \alpha_i)$ for all $i \in I$, $h\in Y$. Then
$\dot{\UU}$ inherits a natural associative $\Q(v)$-algebra structure
from that of $\UU$, as follows: for any $\lambda_1, \lambda_2,
\lambda'_1, \lambda'_2 \in X$ and any $t \in \UU(\lambda_1 -
\lambda'_1)$, $s \in \UU(\lambda_2 - \lambda'_2)$, the product
$\pi_{\lambda_1,\lambda'_1}(t) \pi_{\lambda_2, \lambda'_2}(s)$ is
equal to $\pi_{\lambda_1,\lambda'_2}(ts)$ if $\lambda'_1 = \lambda_2$
and is zero otherwise.

For any $\lambda \in X$, set $1_\lambda = \pi_{\lambda,
  \lambda}(1)$. Then the elements $1_\lambda \in \dot{\UU}$ satisfy
the relations
\[
  1_\lambda 1_{\lambda'} = \delta_{\lambda, \lambda'} 1_\lambda 
\]
and we have ${}_\lambda \UU_{\lambda'} = {1}_\lambda \dot{\UU}
1_{\lambda'}$.  The algebra $\dot{\UU}$ may be regarded as a
$\UU$-bimodule by setting, for $t \in \UU(\nu)$, $s\in \UU$, and $t'
\in \UU(\nu')$, the product $t \pi_{\lambda, \lambda'}(s) t' =
\pi_{\lambda+\nu, \lambda'-\nu'}(tst')$ for any $\lambda, \lambda' \in
X$.  It follows that the products $1_\lambda E_{\pm i}$ ($i \in I$,
$\lambda \in X$) are well defined elements of $\dot{\UU}$. In fact
$\dot{\UU}$ is generated by those elements.

\subsection{} \label{UdotSpi}
For a given saturated subset $\pi$ of $X^+$ we shall write $\pi^c$ for
the set theoretic complement $X^+ - \pi$.  In \cite{PGSA} it is shown
that $\Sch(\pi)$ is isomorphic with the quotient algebra
$\dot{\UU}/\dot{\UU}[\pi^c]$ for any finite saturated subset $\pi$ of
$X^+$.  The ideal $\dot{\UU}[\pi^c]$ and corresponding quotient both
appear in \cite[\S29.2]{Lusztig}; the ideal may be characterized as
the set of all elements $u \in \dot{\UU}$ such that $u$ annihilates
every $\Delta(\lambda)$ with $\lambda \in \pi$. We note for future
reference that
\[
  \textstyle \bigcap_{\pi} \dot{\UU}[\pi^c] = (0)
\]
(see \cite[Chapter~29]{Lusztig}).

The proof of \cite[Theorem~4.2]{PGSA} shows that the quotient map
$\pdot_\pi: \dot{\UU} \to \Sch(\pi)$ (with kernel $\dot{\UU}[\pi^c]$)
is defined by sending $E_{\pm i} 1_\lambda \in \dot{\UU}$ to $E_{\pm
  i} 1_\lambda \in \Sch(\pi)$. Clearly the quotient maps $\pdot_\pi$
fit into a commutative diagram
\[
\xymatrix{
\widehat{\UU} \ar[dr]^{\phat_{\pi'}} \ar@/_/[dddr]_{\phat_\pi} & & 
\dot{\UU} \ar[dl]_{\pdot_{\pi'}} \ar@/^/[dddl]^{\pdot_\pi}
\ar@{.>}[ll]_{\dot{\theta}} \\ 
& \Sch(\pi') \ar[dd]|{f_{\pi,\pi'}} \\ \\
& \Sch(\pi) 
}
\]
for any finite saturated subsets $\pi, \pi'$ of $X^+$ with $\pi
\subset \pi'$. Again the universal property of inverse limits
guarantees the existence of a unique algebra map $\dot{\theta}:
\dot{\UU} \to \widehat{\UU}$ making the diagram commute.

We are now prepared to prove the following result.

\begin{thm} \label{thm:dotUembeds}
The map $\dot{\theta}$ is an algebra embedding of $\dot{\UU}$ into
$\widehat{\UU}$ sending $E_{\pm i} 1_\lambda$ to $\widehat{E}_{\pm i}
\widehat{1}_\lambda$ for all $i \in I$, $\lambda \in X$.  Hence, the
subalgebra of $\widehat{\UU}$ generated by the products
$\widehat{E}_{\pm i} \widehat{1}_\lambda$ ($i \in I$, $\lambda \in X$)
is isomorphic with $\dot{\UU}$.
\end{thm}

\begin{proof}
Suppose $\dot{\theta}(u) = 0$ for $u \in \dot{\UU}$. Then $\pdot_\pi(u)
= 0$ for each finite saturated subset $\pi$ of $X^+$. Hence $u \in
\bigcap_{\pi} \dot{\UU}[\pi^c]$; whence $u = 0$. Thus the kernel of
$\dot{\theta}$ is trivial. The rest of the claims are clear. 
\end{proof}

\subsection{}
Henceforth we identify $\dot{\UU}$ with the subalgebra of
$\widehat{\UU}$ generated by all $\widehat{E}_{\pm i}
\widehat{1}_\lambda$. Note that the elements $\widehat{E}_{\pm i}$ of
$\widehat{\UU}$ are not elements of $\dot{\UU}$ since their expression
in terms of the generators of $\dot{\UU}$ involves infinite sums.

\subsection{} \label{completion}
For convenience, choose a total ordering $\pi_1, \pi_2, \dots$ on the
finite saturated sets $\pi$ which is compatible with the partial order
given by set inclusion, in the following sense: $i \le j$ implies that
$\pi_i \subseteq \pi_j$.  Then the
completion of $\dot{\UU}$ with respect to the descending sequence of
ideals
\[
  \dot{\UU} \supseteq \dot{\UU}[\pi_1^c] \supseteq \dot{\UU}[\pi_2^c]
  \supseteq \cdots
\]
is isomorphic with $\widehat{\UU}$.

As in \cite{RM-Green}, we put a topology on the ring $\dot{\UU}$ by
letting the collection $\{ \dot{\UU}[\pi^c] \}$, as $\pi$ varies over
the finite saturated subsets of $X^+$, define a neighborhood base of
$0$.  Then $\widehat{\UU}$ may be regarded as the set of equivalence
classes of Cauchy sequences $(x_n)_{n=1}^\infty$ of elements of
$\dot{\UU}$ under usual Cauchy equivalence. Here a sequence $(x_n)$ is
Cauchy if for each neighborhood $\dot{\UU}[\pi^c]$ there exists some
positive integer $N(\pi)$ such that
\[
x_m - x_n \in \dot{\UU}[\pi^c] \quad\text{for all $m,n \ge N(\pi)$},
\]
and given sequences $(x_n)$, $(y_n)$ are Cauchy equivalent if $x_n -
y_n \to 0$ as $n \to \infty$.  The proof is standard (see
e.g.\ \cite[Chapter 10]{AM}). Given a Cauchy sequence $(x_n)$ its image
in $\Sch(\pi)$ is eventually constant, say $a_\pi$. The resulting
sequence $(a_\pi) \in \prod_\pi \Sch(\pi)$ satisfies $a_\pi =
f_{\pi,\pi'}(a_{\pi'})$ for any $\pi \subset \pi'$, so $(a_\pi) \in
\widehat{\UU}$. On the other hand, given any $(a_\pi) \in
\widehat{\UU}$ we can define a corresponding Cauchy sequence by
setting $x_n$ equal to any element of the coset
$\pdot_{\pi_n}(a_{\pi_n}) \in \dot{\UU}/\dot{\UU}[\pi_n^c]$, where
$\pi = \pi_n$.

Thus, $\widehat{\UU}$ is a complete topological algebra. It is
Hausdorff, thanks to the triviality of the intersection of the
elements of the neighborhood base of $0$.

\subsection{}
We say that a basis $B$ of $\dot{\UU}$ is {\em coherent} if the set of
nonzero elements of $\pdot_\pi(B)$ is a basis of $\Sch(\pi)$, for each
finite saturated $\pi \subset X^+$.

Assume that $B$ is any such basis. Write $B[\pi]$ for the set of
nonzero elements of $\pdot_\pi(B)$. The following result is a
consequence of \cite[Corollary 2.2.5]{RM-Green}.

\begin{prop}
Given any coherent basis $B$ of $\dot{\UU}$, the completion
$\widehat{\UU}$ may be identified with the algebra of all formal
infinite linear combinations of elements of $B$. 
\end{prop}

\begin{proof}
Any formal sum of the form $a = \sum_{b \in B} a_b\,b$ (for $a_b \in
\Q(v)$) determines an element $a_\pi = \sum_{b \in B[\pi]} a_b\,b$ of
$\Sch(\pi)$. Clearly, the sequence $(a_\pi)$ is an element of
$\widehat{\UU}$. 

We must show that every element of $\widehat{\UU}$ is expressible in
such a form. Let $a = (a_\pi)$ be an element of $\widehat{\UU}$. Each
$a_\pi \in \Sch(\pi)$ may be written in the form $a_\pi = \sum_{b \in
B[\pi]} a_b\,b$ where $a_b \in \Q(v)$. Moreover, the coefficient
$a_b$ of any $b \in B$ will always be the same value, for any $\pi'$
such that $b \in B[\pi']$. To see this, let $\pi''$ be any finite
saturated subset of $X^+$ containing both $\pi$ and $\pi'$ (such must
exist) and consider the projections $f_{\pi, \pi'}$ and $f_{\pi,
\pi''}$. Since $\cup_\pi B[\pi] = B$ this shows that $a$ determines a
well-defined infinite sum $\sum_{b \in B} a_b\, b$.
\end{proof}

\subsection{} \label{can-basis}
In \cite[Chapter~25]{Lusztig} it is proved that the canonical basis
can be lifted from the positive part of $\UU$ to a canonical basis
$\dot{\B}$ of $\dot{\UU}$. (This was a primary motivation for the
introduction of $\dot{\UU}$.)  Moreover, $\dot{\B}$ is coherent with
respect to the inverse system $\{ \Sch(\pi) \}$; see
\cite[\S29.2.3]{Lusztig}. Thus it follows from the preceding
proposition that elements of $\widehat{\UU}$ may be regarded as formal
infinite linear combinations of $\dot{\B}$.

\begin{rmk}
It is easy to see that $\widehat{\UU}$ is a procellular algebra in the
sense of R.M.~Green \cite{RM-Green}. This is a consequence of
Lusztig's refined Peter-Weyl theorem \cite[Theorem~29.3.3]{Lusztig},
which implies that $\dot{\B}$ is a {\em cellular basis} of
$\dot{\UU}$.  (See \cite{GL} for the definition of cellular basis.)
\end{rmk}

\begin{lem}\label{lem:gen}
The algebra $\widehat{\UU}$ is topologically generated by the elements
$\widehat{E}_{\pm i}$ ($i \in I$), $\widehat{1}_\lambda$ ($\lambda \in
X$) in the sense that every element of $\widehat{\UU}$ is expressible
as a formal (possibly infinite) linear combination of finite products
of those elements.
\end{lem}

\begin{proof}
By \ref{can-basis} every element of $\widehat{\UU}$ is a formal linear
combination of elements of $\dot{\B}$. But elements of $\dot{\B}$ are
themselves expressible as finite linear combinations of finite
products of the elements $\widehat{E}_{\pm i}$ ($i \in I$),
$\widehat{1}_\lambda$ ($\lambda \in X$), since $\dot{\UU}$ is
generated by elements of the form $\widehat{E}_{\pm i}
\widehat{1}_\lambda$ for various $i\in I$, $\lambda \in X$ (see
\cite[Chapter 23]{Lusztig}).
\end{proof}

\begin{thm}
The algebra $\widehat{\UU}$ is the associative algebra with 1 given by
the generators $\widehat{E}_{\pm i}$ ($i \in I$),
$\widehat{1}_\lambda$ ($\lambda \in X$) with the relations (a)--(d) of
Proposition~\ref{prop:Uhat-rels}, in the following sense:
\[
\widehat{\UU} \simeq \Q(v)\langle\langle \widehat{E}_{\pm i},
\widehat{1}_\lambda \rangle\rangle / J
\]
where $\Q(v)\langle\langle \widehat{E}_{\pm i}, \widehat{1}_\lambda
\rangle\rangle$ is the free complete algebra on the generators
$\widehat{E}_{\pm i}, \widehat{1}_\lambda$ (consisting of all formal
linear combinations of finite products of generators) and $J$ is the
ideal generated by relations \ref{prop:Uhat-rels}(a)--(d).
\end{thm}

\begin{proof}
Let $\P$ be the algebra $\Q(v)\langle\langle \widehat{E}_{\pm i},
\widehat{1}_\lambda \rangle\rangle / J$. Notice (see
\ref{Spi:reformulation}) that by definition $\Sch(\pi)$ is the
quotient of $\P$ by the ideal generated by all $\widehat{1}_\lambda$
with $\lambda \notin W\pi$. Thus we have surjective quotient maps
\[
q_\pi: \P \to \Sch(\pi)  \qquad \text{($\pi$ finite saturated)}
\]
such that $f_{\pi,\pi'} q_{\pi'} = q_\pi$ whenever $\pi \subset
\pi'$. These maps fit into a commutative diagram similar to the one
appearing in \ref{comm-diag}, and by the universal property of inverse
limits there is an algebra map $\Psi: \P \to \widehat{\UU}$ sending
$\widehat{E}_{\pm i}$ to $\widehat{E}_{\pm i}$, and
$\widehat{1}_\lambda$ to $\widehat{1}_\lambda$.

The map $\Psi$ is injective since the intersection of the kernels of
the various $q_\pi$ is trivial. On the other hand, by the preceding
lemma combined with Proposition~\ref{prop:Uhat-rels} $\Psi$ must also
be surjective, since the generators of $\widehat{\UU}$ satisfy the
defining relations of $\P$.
\end{proof}

\begin{rmk}
The topology on $\widehat{\UU}$ is induced from the topology on
$\dot{\UU}$. The basic neighborhoods of $0$ are of the form
$\widehat{\UU}[\pi^c]$ for the various finite saturated subsets $\pi$
of $X^+$, where $\widehat{\UU}[\pi^c]$ is the set of all formal
$\Q(v)$-linear combinations of elements of $\dot{\B}[\pi^c]$, where
the notation $\dot{\B}[\pi^c]$ is as defined in
\cite[\S29.2.3]{Lusztig}.
\end{rmk}

\section{Integral forms}\noindent
We will now extend the results obtained thus far to integral forms
(over the ring $\A = \Z[v,v^{-1}]$ of Laurent polynomials in $v$).

\subsection{} \label{intSpi}
One has an integral form $_\A\Sch(\pi)$ in $\Sch(\pi)$. It is by
definition the $\A$-subalgebra of $\Sch(\pi)$ generated by all
$\divided{E_{\pm i}}{m}$ ($i \in I$, $m \in \N$) and $1_\lambda$
($\lambda \in W\pi$).  There is an algebra isomorphism
\[
\Sch(\pi) \simeq  \Q(v) \otimes_\A (_\A\Sch(\pi)) 
\]
which carries $\divided{E_{\pm i}}{m}$ to $1 \otimes \divided{E_{\pm
    i}}{m}$ and $1_\lambda$ to $1 \otimes 1_\lambda$.  Note that the
elements $K_h$ ($h \in Y$) in $\Sch(\pi)$ in fact belong to the
subalgebra $_\A \Sch(\pi)$.

It is easy to see (see \cite[\S5.1]{PGSA}) that $_\A\Sch(\pi)$ is
isomorphic with a quotient of the Lusztig $\A$-form $_\A \UU$ of
$\UU$, which is by definition (\cite[\S3.1.13]{Lusztig}) the
$\A$-subalgebra of $\UU$ generated by all $\divided{E_{\pm i}}{m}$ ($i
\in I$, $m \ge 0$) and $K_h$ ($h \in Y$).  The quotient map $_\A\UU
\to {_\A\Sch(\pi)}$ sends $\divided{E_{\pm i}}{m}$ to
$\divided{E_{\pm i}}{m}$ and $K_h$ to $K_h$ ($i \in I$, $m \ge 0$, $h
\in Y$). Hence it is just the restriction of $p_\pi$ to $_\A\UU$; we
denote it also by $p_\pi$.

Clearly the integral form on $\Sch(\pi)$ is compatible with the maps
$f_{\pi,\pi'}$ in the sense that the restriction of $f_{\pi,\pi'}$ to
$_\A\Sch(\pi')$ is a surjective map of $\A$-algebras from
$_\A\Sch(\pi')$ onto $_\A\Sch(\pi)$. Recall the identification
\[
 \widehat{\UU} = \{ (a_\pi)_\pi \in \prod_\pi \Sch(\pi) \mid
 f_{\pi, \pi'}(a_{\pi'}) = a_\pi \text{ whenever } \pi \subset \pi' \}.
\]
Inside this algebra we have an $\A$-subalgebra 
\[
 _\A \widehat{\UU} = \{ (a_\pi)_\pi \in \prod_\pi (_\A \Sch(\pi)) \mid
 f_{\pi, \pi'}(a_{\pi'}) = a_\pi \text{ whenever } \pi \subset \pi' \}.
\]
It is clear that $_\A \widehat{\UU}$ is isomorphic with $\varprojlim
(_\A \Sch(\pi))$ (an isomorphism of $\A$-algebras). 

\begin{thm}
  The map $\theta$ (see \ref{comm-diag}) restricts to an algebra
  embedding (also denoted $\theta$) of $_\A\UU$ into
  $_\A\widehat{\UU}$ sending $\divided{E_{\pm i}}{m} \to
  \divided{\widehat{E}_{\pm i}}{m}$ and $K_h$ to $\widehat{K}_h$ for
  all $i \in I$, $m \ge 0$, $h \in Y$. Hence, $_\A\UU$ is isomorphic
  with the $\A$-subalgebra of $_\A\widehat{\UU}$ generated by all
  $\divided{\widehat{E}_{\pm i}}{m}$, $\widehat{K}_h$ ($i \in I$, $m
  \ge 0$, $h \in Y$).
\end{thm}

\begin{proof}
If $\phat_\pi$ denotes projection to the $\pi$th component as before,
we have a commutative diagram of $\A$-algebras similar to the
commutative diagram considered in \ref{comm-diag}, where the algebras
are replaced by their integral forms and each map is just the
restriction to the integral form of the corresponding map in the
diagram given in \ref{comm-diag}.  The existence of the map $\theta$
is guaranteed by the universal property of inverse limits, and by
considering its effect on generators we see that it must in fact be
the restriction to $_\A\UU$ of the map $\theta$ given already in
\ref{comm-diag}. Since $\theta$ is a restriction of an injective map,
it is itself injective.
\end{proof}

\subsection{}
Now consider the $\A$-subalgebra $_\A\dot{\UU}$ of $\dot{\UU}$
generated by all products of the form $\divided{E_{\pm
    i}}{m}1_\lambda$ ($i \in I$, $m \ge 0$, $\lambda \in X$). This
integral form of $\dot{\UU}$ was studied in \cite[\S23.2]{Lusztig}.

The restriction of the quotient map $\dot{p}_\pi$ (see \ref{UdotSpi})
to $_\A \dot{\UU}$ gives a surjective map (also denoted by
$\dot{p}_\pi$) from $_\A \dot{\UU}$ to $_\A \Sch(\pi)$. This is clear
from the definition of $_\A \Sch(\pi)$ given in \ref{intSpi}.  There
is a commutative diagram similar to the diagram considered in
\ref{UdotSpi}, in which all the algebras are replaced by their
integral forms, and the maps are just the restrictions of the maps
considered in the diagram \ref{UdotSpi}.  As before, the universal
property of inverse limits guarantees the existence of a unique
algebra map $\dot{\theta}: {}_\A\dot{\UU} \to {}_\A\widehat{\UU}$
making the diagram commute.

\begin{thm} \label{thm:AdotUembeds}
The map $\dot{\theta}$ is an algebra embedding of $_\A\dot{\UU}$ into
$_\A\widehat{\UU}$ sending $\divided{E_{\pm i}}{m} 1_\lambda$ to
$\divided{\widehat{E}_{\pm i}}{m} \widehat{1}_\lambda$ for all $i \in
I$, $m\ge 0$, $\lambda \in X$.  Hence, the $\A$-subalgebra of
$_\A\widehat{\UU}$ generated by the $\divided{\widehat{E}_{\pm i}}{m}
\widehat{1}_\lambda$ ($i \in I$, $m\ge 0$, $\lambda \in X$) is
isomorphic with $_\A\dot{\UU}$.
\end{thm}

\begin{proof}
The map $\dot{\theta}$ is the restriction to $_\A\dot{\UU}$ of the
injective map $\dot{\theta}$ considered in the proof of
\ref{thm:dotUembeds}, thus injective.
\end{proof}

\subsection{}
As before, fix a total ordering $\pi_1, \pi_2, \dots$ on the finite
saturated sets $\pi$ compatible with the partial order given by set
inclusion.  Then the completion of ${}_\A\dot{\UU}$ with respect to
the descending sequence of ideals
\[
  {}_\A\dot{\UU} \supseteq {}_\A\dot{\UU}[\pi_1^c] \supseteq
  {}_\A\dot{\UU}[\pi_2^c] \supseteq \cdots
\]
is isomorphic with ${}_\A\widehat{\UU}$.

One may put a topology on $_\A\dot{\UU}$ exactly as in
\ref{completion}, by letting the collection $\{ _\A\dot{\UU}[\pi^c]
\}$, as $\pi$ varies over the finite saturated subsets of $X^+$,
define a neighborhood base of $0$. Here $_\A\dot{\UU}[\pi^c] = {}_\A
\dot{\UU} \cap \dot{\UU}[\pi^c]$ is the kernel of the surjection
$\dot{p}_\pi: {}_\A \dot{\UU} \to {}_\A \Sch(\pi)$.

Since the canonical basis $\dot{\B}$ is an $\A$-basis of $_\A
\dot{\UU}$, it follows that elements of $_\A \widehat{\UU}$ may be
regarded as formal (possibly infinite) $\A$-linear combinations of
$\dot{\B}$.  Then the subalgebra $_\A \dot{\UU}$ may be regarded as
the set of all finite $\A$-linear combinations of $\dot{\B}$.

\section{Specialization} \noindent
By specializing to a commutative ring $R$ (via the ring homomorphism
$\A \to R$ determined by $v \to \xi$ for an invertible $\xi \in R$)
one obtains generalized $q$-Schur algebras $_R\Sch(\pi)$ over $R$ for
each saturated $\pi$. These algebras form an inverse system; we study
the corresponding inverse limit $_R\widehat{\UU}$.

\subsection{} \label{specialization}
Let $R$ be a given commutative ring with $1$, and $\xi \in R$ a
given invertible element. Regard $R$ as an $\A$-algebra via the ring
homomorphism $\A \to R$ such that $v^n \to \xi^n$ for all $n \in
\Z$. Consider the $R$-algebras
\begin{equation}
  _R \UU = R \otimes_\A (_\A \UU), \qquad _R \dot{\UU} = R \otimes_\A
  (_\A \dot{\UU}).
\end{equation}
In the literature, these algebras are sometimes denoted by alternative
notations such as $\UU_\xi$, $\dot{\UU}_\xi$. We note that one has
isomorphisms ${}_{\Q(v)} \UU \simeq \UU$, ${}_{\Q(v)} \dot{\UU} \simeq
\dot{\UU}$ given by the obvious maps.

For a finite saturated subset $\pi$ of $X^+$ set $_R \Sch(\pi) = R
\otimes_\A (_\A \Sch(\pi))$, a generalized $q$-Schur algebra
specialized at $v \to \xi$.  Note that ${}_{\Q(v)} \Sch(\pi) \simeq
\Sch(\pi)$. The elements $\divided{E_{\pm i}}{m}$, $1_\lambda$, $K_h
\in {}_\A\Sch(\pi)$ give rise to corresponding elements
\[
1\otimes \divided{E_{\pm i}}{m},\  1\otimes 1_\lambda,\ 
1\otimes K_h \in {}_R\Sch(\pi) 
\]
for $i\in I$, $m \ge 0$, $\lambda\in X$, $h \in Y$.  These elements of
${}_R\Sch(\pi)$ will be respectively denoted again by
$\divided{E_{\pm i}}{m}$, $1_\lambda$, $K_h$, since the intended
meaning will be clear from the context.

Since tensoring is right exact, we have a
surjective quotient map
\begin{equation}\label{eq:1dotp}
   1 \otimes \dot{p}_\pi: {}_R \dot{\UU} \to {_R \Sch(\pi)},
\end{equation}
arising (by tensoring with the identity map on $R$) from the
corresponding quotient map $\dot{p}_\pi: {}_\A\dot{\UU} \to
{}_\A\Sch(\pi)$ over $\A$.  

\begin{lem}\label{lem:ker1dotp}
The kernel of the map $1\otimes \dot{p}_\pi$ is $_R\dot{\UU}[\pi^c] =
R\otimes_\A ({}_\A \dot{\UU}[\pi^c])$.
\end{lem}

\begin{proof}
  This is a consequence of the fact that the canonical basis of $_\A
  \dot{\UU}$ is a ``cellular'' basis (in the sense of \cite{GL}) and
  the kernel $_\A \dot{\UU}[\pi^c]$ of the map $\dot{p}_\pi:
  {}_\A\dot{\UU} \to {}_\A\Sch(\pi)$ is the cell ideal spanned by the
  canonical basis elements in $\dot{\B}[\pi^c]$. See \cite[\S5]{PGSA}
  for details.
\end{proof}

\subsection{}\label{Uhatdef}
Whenever $\pi \subset \pi'$ (for finite saturated subsets of $X^+$)
there is a surjective algebra map $1\otimes f_{\pi,\pi'}: {}_R
\Sch(\pi') \to {}_R \Sch(\pi)$ obtained from the map $f_{\pi, \pi'}:
    {}_\A \Sch(\pi') \to {}_\A \Sch(\pi)$ by tensoring with the
    identity map on $R$.  Thus we have an inverse system
\begin{equation}\label{eq:RIS}
  \{ _R \Sch(\pi);\ \ 1\otimes f_{\pi, \pi'} \}.
\end{equation}
We define ${}_R \widehat{\UU}$ to be the inverse limit $\varprojlim
{}_R \Sch(\pi)$ of this inverse system. In ${}_R\widehat{\UU}$ we have
elements
\begin{equation}
  \divided{\widehat{E}_{\pm i}}{m}:= (\divided{E_{\pm
      i}}{m})_\pi,\ \widehat{1}_\lambda:=(1_\lambda)_\pi,\ \widehat{K}_h:=
  (K_h)_\pi
\end{equation}
of ${}_R\widehat{\UU}$ defined by the corresponding constant
sequences, for $i\in I$, $m \ge 0$, $\lambda\in X$, $h \in
Y$. Actually, the sequence defining $1_\lambda$ is not necessarily
constant, but it is eventually constant.

\begin{thm}
There is an embedding ${}_R \dot{\UU} \to {}_R\widehat{\UU}$ of
$R$-algebras sending $\divided{E_{\pm i}}{m} 1_\lambda$ to
$\divided{\widehat{E}_{\pm i}}{m} \widehat{1}_\lambda \in
{}_R\widehat{\UU}$ for all $i\in I, m\ge 0$, $\lambda \in X$. Thus
${}_R\dot{\UU}$ may be identified with the $R$-subalgebra of
${}_R\widehat{\UU}$ generated by all $\divided{\widehat{E}_{\pm i}}{m}
\widehat{1}_\lambda$ ($i\in I, m\ge 0$, $\lambda \in X$).
\end{thm}

\begin{proof}
 Consider the commutative diagram of
$R$-algebra maps
\[
\xymatrix{
_R\widehat{\UU} \ar[dr]^{1\otimes\phat_{\pi'}} 
\ar@/_/[dddr]_{1\otimes\phat_\pi} & & 
_R\dot{\UU} \ar[dl]_{1\otimes \dot{p}_{\pi'}} 
\ar@/^/[dddl]^{1\otimes \dot{p}_\pi}
\ar@{.>}[ll]_{_R\dot{\theta}} \\ 
& _R\Sch(\pi') \ar[dd]|{1\otimes f_{\pi,\pi'}} \\ \\
& _R\Sch(\pi) 
}
\]
for any finite saturated $\pi \subset \pi'$.  The universal property
of inverse limits guarantees the existence of a unique algebra map
$_R\dot{\theta}: {}_R\dot{\UU} \to {}_R\widehat{\UU}$ making the
diagram commute.  This map has the desired properties.

By Lemma \ref{lem:ker1dotp}, the kernel of $_R\dot{\theta}$ is the
intersection over $\pi$ of all ${}_R\dot{\UU}[\pi^c]$.  From Lusztig's
results \cite[Chapter 29]{Lusztig} it follows that this intersection
is the zero ideal $(0)$.  Indeed, the intersection is contained in the
intersection of all ${}_R\dot{\UU}[\ge \lambda]$ as $\lambda$ runs
through all dominant weights, and by known properties of the canonical
basis the latter intersection is $(0)$.  This proves the injectivity
of $_R\dot{\theta}$, as desired.
\end{proof}

\subsection{}
As before, fix a total ordering $\pi_1, \pi_2, \dots$ on the finite
saturated sets $\pi$ which is compatible with the partial order given
by set inclusion.  Then the completion of ${}_R\dot{\UU}$ with
respect to the descending sequence of ideals
\[
  {}_R\dot{\UU} \supseteq {}_R\dot{\UU}[\pi_1^c] \supseteq
  {}_R\dot{\UU}[\pi_2^c] \supseteq \cdots
\]
is isomorphic with ${}_R\widehat{\UU}$.

One may put a topology on $_R\dot{\UU}$, by letting the collection $\{
_R \dot{\UU}[\pi^c] \}$, as $\pi$ varies over the finite saturated
subsets of $X^+$, define a neighborhood base of $0$.  Elements of $_R
\widehat{\UU}$ may be regarded as formal (possibly infinite)
$R$-linear combinations of $\dot{\B}$.  Then the subalgebra $_R
\dot{\UU}$ may be regarded as the set of all finite $R$-linear
combinations of $\dot{\B}$.  The topology on $_R \widehat{\UU}$ is
induced from the topology on $_R \dot{\UU}$; i.e., the basic
neighborhoods of $0$ are of the form $_R\widehat{\UU}[\pi^c]$ for the
various finite saturated sets $\pi$, where $_R\widehat{\UU}[\pi^c]$ is
the set of all formal $R$-linear combinations of elements of
$\dot{\B}[\pi^c]$.

\begin{rmk}
It is not immediately clear how the completion $_R\widehat{\UU}$ is
related to $R \otimes_\A ({}_\A\widehat{\UU})$. We note that there is
a well defined homomorphism of $R$-algebras
\begin{equation}
  R \otimes_\A ({}_\A\widehat{\UU}) \to {}_R\widehat{\UU}
\end{equation}
sending $1 \otimes (a_\pi)$ to $(1 \otimes a_\pi)$, where $(a_\pi) \in
\prod_\pi ({}_\A\Sch(\pi))$ satisfies the condition
$f_{\pi',\pi}(a_{\pi'}) = a_\pi$ whenever $\pi \subset \pi'$. It seems
unlikely that this map is an isomorphism.
\end{rmk}

\begin{rmk}
Does the algebra ${}_R\UU$ embed in ${}_R\widehat{\UU}$? There is an
$R$-algebra homomorphism
\begin{equation}
  {}_R \UU \to {}_R\widehat{\UU}
\end{equation} 
defined on generators by sending $\divided{E_{\pm i}}{m}$ to
$\divided{\widehat{E}_{\pm i}}{m}$ (for $i \in I$, $m \ge 0$), and
sending $K_h$ to $\widehat{K}_h$ (for $h \in Y$). But it is not clear
that this map is injective.

One would like to prove the injectivity of this map, since then one
may identify ${}_R\UU$ with a subalgebra of the completion
${}_R\widehat{\UU}$.  

We sketch a possible approach to this question.
To prove the injectivity, one needs to show that the intersection of
the kernels of the quotient maps ${}_R\UU \to {}_R\Sch(\pi)$ is
$(0)$. The quotient map ${}_R\UU \to {}_R\Sch(\pi)$ is the map
$1\otimes p_\pi$ obtained from the quotient map $p_\pi: {}_\A\UU \to
{}_\A\Sch(\pi)$ defined in \ref{intSpi}, by tensoring with the
identity map on $R$.

Suppose some $u \in {}_R\UU$ belongs to the intersection of the
kernels of the quotient maps ${}_R\UU \to {}_R\Sch(\pi)$, as $\pi$
varies over the finite saturated subsets of $X^+$. Then one can show
that $u$ acts as zero on any finite-dimensional ${}_R\UU$-module,
since such a module will be a well defined module for ${}_R\Sch(\pi)$
for some (large enough) saturated set $\pi$. Now if \cite[Proposition
  5.11]{Jantzen:LQG} can be generalized to our setting, we would be
able to conclude that $u=0$, and the desired injectivity statement
would be established. However, such a generalization is not available
in the published literature on quantum groups, to the author's
knowledge.  
\end{rmk}


\end{document}